\documentclass[11pt,leqno]{article}
\usepackage{graphicx, amsfonts, amsthm, amsxtra, amssymb, verbatim,latexsym,bbm}
\usepackage{calrsfs}
\usepackage{addfont}
\usepackage{hyperref}
\usepackage{color}   
\usepackage[mathscr]{euscript}
\usepackage[all,cmtip]{xy}
\usepackage{pifont}

\begin{document}

\def\Z{\mathbb{Z}}                   
\def\Q{\mathbb{Q}}                   
\def\C{\mathbb{C}}                   
\def\N{\mathbb{N}}                   
\def\R{\mathbb{R}}                   

\def\Ra{{\sf R}}                      
\def\T{{\sf T}}                      
\def\t{{\sf t}}                       

\def\tr{{{\mathsf t}{\mathsf r}}}                 

\def\rc{{Rankin-Cohen }}
\def\rs{{Ramanujan-Serre type derivation }}
\def\DHR{{\rm DHR }}            
\def\H{{\sf{ E}}}
\def\P {\mathbb{P}}                  
\def\E{{\cal E}}
\def\L{{\sf L}}             
\def\CX{{\cal X}}
\def\dt{{\sf d}}             
\def\LG{{\sf G}}   
\def\LA{{\rm Lie}(\LG)}   
\def\amsy{\mathfrak{G}}  
\def\gG{{\sf g}}   
\def\gL{\mathfrak{g}}   
\def\rvf{{ \mathfrak{H}}}   
\def\cvf{{\mathfrak{F}}}   
\def\di2{{ m}}   
\def\a{\mathfrak{a}}
\def\b{\mathfrak{b}}

\def\mmat{e}
\def\cmat{f}
\def\rmat{h}
\def\sl2{\mathfrak{sl}_2(\C)}
\def\SL2{{\rm SL}_2(\Z)}
\def\qmfs{\widetilde{\mathscr{M}}}             
\def\mfs{{\mathscr{M}}}                            
\def\cyqmfs{{ \mathcal{M}}}             
\def\cymfs{{\mathcal{M}^2}}                            
\def\Ramvf{{\sf Ra}}             
\def\rcv{{\sf D}}             
\def\rcdo{{\mathscr{D}}}             
\def\rcdomf{{\mathfrak{D}}}
\def\mdo{{\mathcal{R}}}             
\def\rdo{{\mathscr{W}}}             
\def\cdo{{\delta}}             
\def\rsdo{\mathfrak{d}}             
\def\rstdo{{\partial}}             

\def\G0g{{\Gamma_0(N_g)}}             
\def\mdg{{\Delta_g}}             
\def\cmdg{{\delta_g}}             
\def\qmfg{{{E_g}_2}}             
\def\mfg{{\mathcal{E}_g}}             
\def\wmdg{{w_g}}             

\def\fvs{{t_1}}             
\def\svs{{t_2}}             
\def\tvs{{t_3}}             
\def\dfvs{{R^1}}             
\def\dsvs{{R^2}}             
\def\dtvs{{R^3}}             
\def\fss{{{\sf t}_1}}             
\def\sss{{{\sf t}_2}}             
\def\tss{{{\sf t}_3}}             
\def\wfss{{{\sf w}_1}}             
\def\wsss{{{\sf w}_2}}             
\def\wtss{{{\sf w}_3}}             
\def\wssj{{{\sf w}_j}}             
\def\wss{{{\sf w}}}             

\def\md{{\Delta}}             
\def\wmd{{{\sf w}}}             
\def\mf{{\sf j}}             
\def\bmf{{\sf t}}             
\def\vs{{t}}             
\def\prs{\mathscr{P}}             
\def\nvs{{d}}             

\def\k{{\sf k}}                     

\def\qme{{\cal{E}}_2}             
\def\ew4{{\cal{E}}_4}             
\def \zn{{\frac{1}{N}\Z}}
\def \KKder{\theta}   
\def \KK{\vartheta}   

\def\cyqmfs{{ \widetilde{\mathcal{M}}}}             
\def\cymfs{{\mathcal{M}}}                            

\def\mdo{{\mathcal{R}}}             

\newtheorem{theo}{Theorem}[section]

\newtheorem{cor}{Corollary}[section]
\newtheorem{prob}{Problem}[section]
\newtheorem{lemm}{Lemma}[section]
\newtheorem{prop}{Proposition}[section]
\newtheorem{obs}{Observation}[section]
\newtheorem{conj}{Conjecture}
\newtheorem{nota}{Notation}[section]
\newtheorem{ass}{Assumption}[section]

\theoremstyle{definition}
\newtheorem{exam}{Example}[section]
\newtheorem{defi}{Definition}[section]
\newtheorem{rem}{Remark}[section]

\newtheorem{calc}{}
\numberwithin{equation}{section}

\begin{center}
{\LARGE\bf About quasi-modular forms, differential operators and Rankin--Cohen algebras }
\\

\vspace{.25in} {\large {\sc Younes Nikdelan \footnote{Departamento de An\'alise Matem\'atica, Instituto de Matem\'atica
e Estat\'{i}stica (IME), Universidade do Estado do
Rio de Janeiro (UERJ), Rua S\~{a}o Francisco Xavier, 524, Rio de Janeiro, Brazil / CEP: 20550-900. e-mail: younes.nikdelan@ime.uerj.br} }} \\
\end{center}
\vspace{.25in}
\begin{abstract}
We establish sufficient conditions, involving Rankin–Cohen (RC) brackets, under which certain combinations of meromorphic quasi-modular forms and their derivatives yield meromorphic modular forms. To achieve this, we adopt an algebraic perspective by working within the framework of RC algebras. First, we prove that any canonical RC algebra, whose underlying graded algebra is of type $\frac{1}{N}\mathbb{Z}$ with $N\in \mathbb{N}$, is a sub-RC algebra of a standard RC algebra. We then present and prove an algebraic formulation of results stating that specific combinations of the quasi-modular form $E_2$ with either other modular forms or itself, along with their derivatives, result in modular forms. Next, we provide equivalent formulations of these results in terms of the Lie algebra $\mathfrak{sl}_2(\mathbb{C})$ and Ramanujan systems of RC type. Finally, we discuss applications not only to meromorphic quasi-modular forms but also to Calabi–Yau quasi-modular forms.
\end{abstract}

\section{Introduction} \label{section introduction}

In the theory of quasi-modular forms, developing methods to construct new meromorphic modular forms from a given set of meromorphic  (quasi-)modular forms is of significant importance. One approach to this involves the use of derivatives of quasi-modular forms, which naturally connect to differential equations in other fields like string theory, quantum gravity, and conformal field theory. In these contexts, modular forms often emerge as solutions that encode deep arithmetic and geometric properties, frequently unveiling hidden symmetries within physical theories.

It is well known that the derivative of a modular form is a quasi-modular form, which is not necessarily modular. However, it has been established that for a given modular form $f$ of weight $k$ for $\SL2$, the Serre derivative, defined as
\begin{equation}\label{eq RS der intro}
  \rsdo f:=f'-\frac{k}{12}E_2f,
\end{equation}
where $E_2$ is the Eisenstein series of weight $2$, produces a modular form of weight $k+2$. Recall that $E_2$ is a quasi-modular form which is not modular. Additionally, the polynomial expression $kff''-(k+1)f'^2$ yields a modular form of weight $2k+4$. The latter polynomial relation was generalized by Rankin in \cite{ran56}, who described necessary conditions under which a polynomial relation involving a modular form and its derivatives remains a modular form. Cohen \cite{coh75} extended Rankin’s results, defining, for any non-negative integer $n$, a bilinear operator $F_n(\cdot,\cdot)$. He proved that for any modular forms $f$ and $g$ of weights $k$ and $l$, respectively, $F_n(f,g)$ is a modular form of weight $k+l+2n$. Zagier, in \cite{zag94}, referred to these bilinear operators as the \emph{Rankin--Cohen (RC) brackets}, denoted by $[\cdot,\cdot]_n$ (see \eqref{eq rcb mf}). Furthermore, he developed the theory of \emph{RC algebras} (see Section \ref{section RCa}), and in particular introduced the standard RC and canonical RC algebras, whose underlying space is a connected graded algebra of type $\Z_{\geq 0}$. Notably, he proved that any canonical RC algebra, with underlying connected graded algebra of type $\Z_{\geq 0}$, is a sub-RC algebra of a standard RC algebra.

The space of meromorphic quasi-modular forms can be structured within a graded algebra of type $\zn$ for some $N\in \N$. Additionally, in \cite{younes3}, the author investigated the theory of Calabi-Yau (CY) (quasi-)modular forms associated with the Dwork family, whose underlying graded algebra is of type $\Z$. The space of these forms were equipped with both standard and canonical algebraic RC structures. To demonstrate that the canonical algebraic RC structure forms a sub-RC algebra of the standard structure led the author to conjecture that Zagier's result generalizes to canonical RC algebras of type $\Z$. The first main result of the present paper addresses this conjecture, proving a stronger version for canonical RC algebras of type $\frac{1}{N}\mathbb{Z}$, where $N \in \N$.

\begin{theo} \label{thm1}
Any canonical RC algebra, with underlying graded algebra of type $\zn$, is a sub-RC algebra of a standard RC algebra.
\end{theo}

The main ingredient in the proof of this theorem is the Cohen-Kuznetsov series, with the primary challenge being its definition for elements of non-positive integer grading. A more detailed version of this theorem, which outlines the construction of the desired standard RC algebra, is given in Theorem \ref{thm main1} and will be proved in Section \ref{section proof 1}.

Note that while standard RC algebras serve as typical examples of RC algebras, the same is not immediately clear for canonical RC algebras. However, an immediate corollary of Theorem \ref{thm1} establishes that any canonical RC algebra is also an RC algebra.

RC bracket of modular forms gives a polynomial relation between two modular forms and their derivatives, which results in another modular form. Two natural questions arise: (1) is there a similar polynomial relation involving a modular form, a quasi-modular form, and their derivatives that produces a new modular form? (2) Does a similar result hold for two quasi-modular forms?

We previously mentioned one such polynomial relation when discussing the Serre derivation in \eqref{eq RS der intro}. Kaneko and Koike \cite{KK06} extended this idea, proving that for any non-negative integer $n$ and any modular form $f$ of weight $k$ for $\SL2$, the expression
\begin{equation}\label{eq KK der}
  \KKder^{(n)}f:=f^{(n+1)}-\frac{n+k}{12}[E_2,f]_n,
\end{equation}
is a modular form of weight $k+2+2n$. In this context, \eqref{eq KK der} expresses a polynomial relation involving the quasi-modular form $E_2$, the modular form $f$, and their derivatives. For $n=0$, the relation \eqref{eq KK der} reduces to the Serre derivation given in \eqref{eq RS der intro}. Besides this, If for any non-negative integer $n$ we define:
 \begin{equation}\label{eq KK der E_2}
  \KKder^{(n)}E_2:=\big( 1+(-1)^n \big)E_2^{(n+1)}-\frac{n+2}{12}[E_2,E_2]_n,
\end{equation}
Cohen and  Str\"{o}mberg, in \cite[Proposition 5.3.27]{cs17}, proved that $\KKder^{(n)} E_2$ is a modular form of weight $4+2n$ for $\SL2$. This result directly addresses the second question posed above. Observe that the equation \eqref{eq KK der E_2} for $n=2$ reduces to the well-known Chazy equation
\begin{equation}\label{eq Chazy mf}
  2y'''-2yy''+3(y')^2=0.
\end{equation}

The second main result of this paper seeks to generalize the results related to \eqref{eq KK der} and \eqref{eq KK der E_2} within a broader algebraic framework, and prove them without relying on the established properties of modular forms or the quasi-modular form $E_2$. To state the next main theorem, let $M$ and $\widetilde{M}$ be the canonical and standard RC algebras discussed in Theorem \ref{thm1}, respectively, equipped with RC brackets $[\cdot , \cdot]_{\rstdo,\ew4,\ast}$ and $[\cdot , \cdot]_{\rcdo,\ast}$. Here, $\rcdo$ and $\rstdo$ are derivations on $\widetilde{M}$ and $M$, respectively, while $\ew4\in M$ is an element of weight $4$. In the proof of Theorem \ref{thm1}, we observe the existence of a weight-2 element $\qme\in \widetilde{M}\setminus M$ such that $\widetilde{M}=M[\qme]$ and the restriction of the RC bracket satisfies $[\cdot , \cdot]_{\rcdo,\ast}|_{M\otimes M}=[\cdot , \cdot]_{\rstdo,\ew4,\ast}$. The element $\qme$ plays a role analogous to the quasi-modular form  $E_2$. Recall that, for the spaces of quasi-modular forms $\qmfs(\SL2)$ and  modular forms $\mfs(\SL2)$ associated with $\SL2$, it holds that $\qmfs(\SL2)=\mfs(\SL2)[E_2]$.

\begin{theo} \label{thm2}
For any non-negative integer $n$ and any homogeneous element $f\in M$ of weight $k$,  define:
\begin{align}
  &\KK^{(n)}f:=\rcdo^{n+1}f-(n+k)[\qme,f]_{\rcdo,n}, \label{eq KK RC alg}\\
  &\KK^{(n)}\qme:=\big( 1+(-1)^n \big)\rcdo^{n+1}\qme-(n+2)[\qme,\qme]_{\rcdo,n}. \label{eq KK E_2 RC alg}
\end{align}
Then, $\KK^{(n)}f$ and $\KK^{(n)}\qme$ belong to $M$ and are elements of weight $k+2+2n$ and $4+2n$, respectively.
\end{theo}

Both Theorem \ref{thm1} and Theorem \ref{thm2} have significant applications, and we highlight a few that have already appeared in the author's other contributions.

The most notable application, and the one that perhaps holds the greatest importance as the origin of Theorem \ref{thm1}, is an extension of the main result presented in \cite{younes3}. In that work, we introduced the spaces of CY quasi-modular forms, $\cyqmfs$, and CY modular forms, $\cymfs$, associated with the Dwork family, equipping $\cyqmfs$ with an RC algebraic structure. It is worth noting that the underlying graded algebra of $\cyqmfs$ is of type $\Z$, and $\cymfs \subset \cyqmfs$. However, due to the absence of Theorem \ref{thm1}, in \cite[Theorem 1.2]{younes3} we were only able to prove that the space of CY modular forms of \emph{positive weight} is closed under the RC brackets of the CY quasi-modular forms (see Example \ref{ex cyqms}). Now, thanks to Theorem \ref{thm1}, we can establish an extension of this result, using the same proof method presented in \cite{younes3}.

\begin{theo} \label{thm3}
{\bf(extension of \cite[Theorem 1.2]{younes3})}
The space of CY modular forms associated with the Dwork family is closed under the Rankin-Cohen brackets of the CY quasi-modular forms associated with the Dwork family.
\end{theo}

Theorem \ref{thm3} is, in fact, an analogous result to the aforementioned property  in the theory of quasi-modular forms proved by Cohen, which states that the space of modular forms is closed under the RC bracket $[\cdot, \cdot]_\ast$ defined on the space of quasi-modular forms.

The author and Bogo in \cite{bn23} also refer to the results of the present paper to establish their contributions. One of the main results of that work is of crucial significance for this paper as well, as it asserts that the converse of Theorem \ref{thm1} is, in some sense, valid. Therefore, we state it here without proof, referring readers to \cite{bn23} for further details. In this theorem, the weight operator (derivation) $\rdo$ refers to the operator that multiplies any homogeneous element $f$ of degree $k$ by its degree, i.e., $\rdo f = k f$.

\begin{theo} \label{thm4}
{\bf  (\cite[Theorem 2.2]{bn23})}
A finitely generated RC algebra $(M,[\,,]_\ast)$ is a canonical RC algebra if and only if it is a sub-RC algebra of a standard RC algebra $(\widetilde{M}:=M[\qme],[\cdot,\cdot]_{\rcdo,\ast})$ with the following property: the derivation $\rcdo$, along with the weight operator $\rdo$ and a derivation~$\cdo$ satisfying~$\ker  \cdo=M$, endows $\widetilde{M}$ with an $\sl2$-module structure, i.e., the Lie bracket relations $[\rcdo,\cdo]=\rdo,\ [\rdo,\rcdo]=2\rcdo,\ [\rdo,\cdo]=-2\cdo$ hold.
\end{theo}

Recall that in the theory of quasi-modular forms for $\SL2$, the usual derivation $\rcdomf$, together with the weight operator  $\rdo$ and $\cdo=-12\frac{\partial }{\partial E_2}$, endows $\qmfs(\SL2)$ with an $\sl2$-module structure. A similar result holds for the space of quasi-modular forms for any non-cocompact Fuchsian subgroup $\Gamma\subset {\rm PSL}_2(\R)$ (See Section \ref{section sl2}).

In \cite{younes5}, the author applied Theorem \ref{thm2} to derive the relations \cite[(7.20) and (7.31)]{younes5}, from which several interesting congruences for Ramanujan-type tau functions in $\Gamma_0(2)$ and $\Gamma_0(3)$ were deduced. For further details, see \cite[\S 7.3 and \S 7.4]{younes5}.

\begin{rem}
The main results of this work hold when the underlying graded algebra is of type $\Q$, and the same proofs apply.
\end{rem}

The structure of this paper is as follows. In Section \ref{section NT}, we introduce our notations and terminology. Section \ref{section RCa} is dedicated to establishing the main definitions and presenting a more detailed version of the first main result, Theorem \ref{thm1}. In Section \ref{section CC series}, we prepare the key components of the proof of Theorem \ref{thm1}, including the Cohen-Kuznetsov series and several important technical lemmas. Section \ref{section proof 1} presents the proof of Theorem \ref{thm1} (or equivalently, Theorem \ref{thm main1}). Section \ref{section ctlII} addresses additional significant technical lemmas, which will be used to prove Theorem \ref{thm2} in Section \ref{section proof 2}. Finally, in Section \ref{section sl2}, we explore the relationship between the converse of Theorem \ref{thm1}, the Lie algebra $\sl2$, and Ramanujan systems of RC type. We also present a non-classical example in this section, and conclude it by noting that any finitely generated RC algebra containing a homogeneous element that is not a zero divisor can be viewed as a sub-RC algebra of a standard RC algebra.


\section{Notations and terminologies} \label{section NT}

Let us first establish the following notations and terminology, which will be used throughout.
\begin{description}
  \item[$\bullet$] $\qmfs(\Gamma)$ and $\mfs(\Gamma)$ refer to the graded algebra of quasi-modular forms and modular forms on a Fuchsian group $\Gamma\subset {\rm PSL}_2(\R)$,  respectively. Specifically, $\qmfs(\Gamma)=\bigoplus_{k=0}^\infty\qmfs_k(\Gamma)$ and $\mfs(\Gamma)=\bigoplus_{k=0}^\infty\mfs_k(\Gamma)$, where $\qmfs_k(\Gamma)$ and $\mfs_k(\Gamma)$ represent the spaces of quasi-modular forms and modular forms of weight $k$ on $\Gamma$, respectively.
\item[$\bullet$] For $\Gamma=\SL2$, it is known that $\qmfs(\SL2)=\C[E_2,E_4,E_6]$ and $\mfs(\SL2)=\C[E_4,E_6]$, where $E_4$ and $E_6$ are modular forms (Eisenstein series) of weights $4$ and $6$, respectively.
  \item[$\bullet$] The derivation $\rcdomf = '$ refers to the normalized usual derivation of quasi-modular forms.
Let us illustrate this in the case where $\Gamma$ is a subgroup of finite index in $\SL2$ and { \tiny $\left( {\begin{array}{*{20}c}
   {1} & {0}   \\
   { 0}  & {1}   \\
\end{array}} \right)$}$\in \Gamma$. In this setting, any quasi-modular form $f$ admits a $q$-expansion (Fourier series) of the form
\[
f(\tau) = f(q) = \sum_{n=0}^{\infty} a_n(f) q^n,
\]
where $a_n(f) \in \mathbb{C}$, $\tau \in \mathbb{H} := \{z \in \mathbb{C} \mid \mathrm{Im}(z) > 0\}$, and $q = e^{2\pi i \tau}$.
Then, the derivation satisfies the relation:
\begin{equation} \label{eq der of mf}
  \rcdomf f = f' = \frac{1}{2\pi i} \frac{df}{d\tau} = q \frac{df}{dq}.
\end{equation}
  \item[$\bullet$] $\k$ denotes a field of characteristic zero.
  \item[$\bullet$] A graded algebra of type $\zn$ refers to a commutative, associative graded algebra with unit of type $\zn$ over $\k$, denoted by $M =\bigoplus_{k\in \zn}M_k$, satisfying $\dim_\k M_k<\infty$, for all $k\in \zn$. Here, $M_k$ consists of homogeneous elements of degree $k$ in $M$, and for any $f \in M_k$, we simply say that $f$ is an element of $M$ of degree $k$. We also use the notation $\widetilde{M} = \bigoplus_{k \in \zn} \widetilde{M}_k$ for a graded algebra of type $\zn$, where $M \subset \widetilde{M}$. Recall that a graded algebra is connected if $M_0 = \k \cdot 1$. For example, $\qmfs(\SL2)$ and $\mfs(\SL2)$ are connected graded algebras of type $\Z_{\geq 0}$.
  \item[$\bullet$] A derivation $\rcdo$ on the graded algebra $M=\bigoplus_{k\in \zn}M_k$ is said to be of degree $m\in \Z$ if $\rcdo(M_k)\subseteq M_{k+m}$, for all $k\in \zn$, which will be denoted by $\rcdo: M_\ast \to M_{\ast+m}$. Furthermore,  $\rcdo^jf$ denotes the $j$th derivative of $f\in M$ with respect to the derivation $\rcdo$. For instance, $\rcdomf$ is a degree-2 derivation on $\qmfs(\SL2)$.
\end{description}


\section{Rankin-Cohen algebra} \label{section RCa}
For all $n\in \Z_{\geq 0}$ and for any two modular forms $f\in \qmfs_{k}(\Gamma)$ and $g\in \qmfs_{l}(\Gamma)$, where $k, l\in \Z_{\geq 0}$, the $n$th Rankin-Cohen (RC) bracket $[f,g]_n$ is defined as the following bilinear operator:
\begin{align}\label{eq rcb mf}
  [f,g]_n:&=\sum_{i+j=n}(-1)^j\binom{n+k-1}{i}\binom{n+l-1}{j}f^{(j)}g^{(i)}\\
  &=\sum_{j=0}^n(-1)^j\frac{(k+j)_{n-j}(l+n-j)_j}{j!(n-j)!} f^{(j)}g^{(n-j)}  \nonumber \,,
\end{align}
where $f^{(j)}$ refers to the $j$th derivative of the (quasi-)modular form $f$, $(x)_0=1$, and  $(x)_n:=x(x+1)\ldots (x+n-1)$ is the Pochhammer symbol. Cohen \cite{coh75} showed that if $f\in \mfs_{k}(\Gamma)$ and $g\in \mfs_{l}(\Gamma)$, then $[f,g]_n\in \mfs_{k+l+2n}(\Gamma)$.

In \cite{zag94}, Zagier listed 11 properties of the RC brackets of modular forms, including the following:
\begin{align*}
  &[f,g]_n =(-1)^n[g,f]_n\,, \ \ \forall \, n\geq 0\,,\\  
  &[[f,g]_0,h]_0 =[f,[g,h]_0]_0\,, \\  
  &[f,1]_0 =[1,f]_0=f\,, \ \  [f,1]_n =[1,f]_n=0\,, \ \forall \,n>0\,,\\ 
  &[[f,g]_1,h]_1+[[g,h]_1,f]_1+[[h,f]_1,g]_1=0\,, 
\end{align*}
in which $f\in \mfs_{k}(\Gamma)$, $g\in \mfs_{l}(\Gamma)$ and $h\in \mfs_{m}(\Gamma)$, with $k, l, m\in \Z_{\geq 0}$.
Zagier also  defined the algebraic RC structures on a commutative, associative graded algebra with unit $M =\bigoplus_{k\geq 0}M_k$ of type $\Z_{\geq 0}$ over the field $\k$ satisfying $M_0=\k.1$ and $\dim_\k M_k<\infty$, for all $k\in \Z_{\geq 0}$. We extend his definition to the graded algebras of type $\zn$.

\begin{defi} \label{defi rca}
An \emph{RC algebra} is a pair $(M,[\cdot,\cdot]_\ast)$, where:
\begin{itemize}
  \item $M =\bigoplus_{k\in \zn}M_k$ is a graded algebra of type $\zn$ over $\k$;
  \item $[\cdot,\cdot]_\ast$ denotes a family of bilinear operations $[ \ , \ ]_n:M_k\otimes M_l\to M_{k+l+2n}$ for $k, l\in \zn$ and $n\geq 0$, satisfying all the algebraic identities associated with the RC brackets of modular forms given in \eqref{eq rcb mf}.
\end{itemize}
\end{defi}

Note that $M_0$ is not necessarily composed solely of constant elements (i.e., elements of the field $\k$). Furthermore, in verifying the algebraic identities, the degree of elements in $M$ may vary within $\zn$.

Regarding Definition \ref{defi rca}, Zagier, in \cite[\S 5]{zag94}, states:
\begin{quote}
``\ldots this may seem like a strange definition, since
we do not know how to give a complete set of axioms. Nevertheless, we will be able
to construct examples and, to a large extent, to clarify the structure of these objects. The situation should be thought of as analogous to building up the theory of Lie
algebras \ldots. In the same way, we will start by considering RC algebras which are subspaces
closed under all bracket operations of some standard examples, and then show that,
under some general hypotheses, all RC algebras in fact arise in this way.''
\end{quote}
We conclude this paper in Section \ref{section sl2} by observing that any finitely generated RC algebra that contains a homogeneous element, which is not a zero divisor, can be regarded as a sub-RC algebra of a standard RC algebra.

In the following we give the definition of standard RC algebras, which in the case of graded algebras of type $\Z_{\geq 0}$ was first treated in \cite{zag94}, and are basic examples of RC algebras.

\begin{defi} \label{defi srca}
An RC algebra $\big(\widetilde{M},[\cdot,\cdot]_\ast\big)$, where $\widetilde{M}=\bigoplus_{k\in \frac{1}{N}\Z}\widetilde{M}_k$ is a graded algebra of type $\frac{1}{N}\Z$, is called a \emph{standard RC algebra} if there exists a degree-2 derivation $\rcdo: \widetilde{M}_\ast \to \widetilde{M}_{\ast+2}$ on $\widetilde{M}$ such that $[\cdot,\cdot]_n=[\cdot,\cdot]_{\rcdo,n}$, for all $n\geq 0$, in which, for given $f\in \widetilde{M}_k$ and $g\in \widetilde{M}_l$, with $k,l\in \zn$, the bracket $[f,g]_{\rcdo,n}$ is defined by:
\begin{equation}\label{eq srcb}
  [f,g]_{\rcdo,n}:=\sum_{j=0}^n(-1)^j\frac{(k+j)_{n-j}(l+n-j)_j}{j!(n-j)!} \rcdo^{j}f \rcdo^{n-j}g.
\end{equation}
\end{defi}

In the RC brackets \eqref{eq rcb mf} of the graded algebra $\mfs(\Gamma)$ of the modular forms, the usual derivation $\rcdomf$ given in \eqref{eq der of mf} is applied, i.e., $[\cdot , \cdot]_{\ast}=[\cdot , \cdot]_{\rcdomf,\ast}$. However, $(\mfs(\Gamma),[\cdot , \cdot]_{\rcdomf,\ast})$ does not constitute a standard RC algebra because $\mfs(\Gamma)$ is not closed under the derivation $\rcdomf$. Instead, it forms a sub-RC algebra of the standard RC algebra $(\qmfs(\Gamma),[\cdot , \cdot]_{\rcdomf,\ast})$. Moreover, its RC brackets can be revisited through a bilinear operator using the Serre derivation. To elucidate these concepts further, we consider the specific case $\Gamma=\SL2$. It is well known that $\qmfs(\SL2)=\C[E_2,E_4,E_6]$ and $\mfs(\SL2)=\C[E_4,E_6]$, where  $E_2,E_4,E_6$ are Eisenstein series. Specifically, $E_2$ is a quasi-modular form of weight 2, while $E_4$ and $E_6$ are modular forms of weights 4 and 6, respectively. As we will see in Section \ref{section sl2}, the derivation $\rcdomf$ can be expressed for any $f\in \qmfs_k(\SL2)$, with $k\in \Z_{\geq 0}$, as follows:
\begin{equation}\label{eq qmf der}
  \rcdomf f=f'=\frac{E_2^2-E_4}{12}\frac{\partial f}{\partial E_2}+\frac{E_2E_4-E_6}{3}\frac{\partial f }{\partial E_4}+\frac{E_2E_6-E_4^2}{2}\frac{\partial f}{\partial E_6}.
\end{equation}
We observe that $\rcdomf$ is a degree-2 derivation, hence $(\qmfs(\SL2),[\cdot , \cdot]_{\rcdomf,\ast})$ forms a standard RC algebra.  From  \eqref{eq qmf der} it is easily seen that if $f$ is a modular form, then $\rcdomf f$ is not necessarily a modular form. Thus,  $(\mfs(\SL2),[\cdot , \cdot]_{\rcdomf,\ast})$ is not a standard RC algebra, but it is a sub-RC algebra of the standard RC algebra $(\qmfs(\SL2),[\cdot , \cdot]_{\rcdomf,\ast})$. If, instead of $\rcdomf$, we consider the Serre derivation $\rsdo$, which for all $f\in \mfs_k(\SL2), \ k\in \Z_{\geq 0}$, is defined by:
\begin{equation}\label{eq rsd}
  \rsdo f=f'-\frac{1}{12}kE_2f=-\frac{E_6}{3}\frac{\partial f}{\partial E_4}-\frac{E_4^2}{2}\frac{\partial f}{\partial E_6}\in \mfs_{k+2}(\SL2)\,,
\end{equation}
then we observe that $\mfs(\SL2)$ is closed under the degree-2 derivation $\rsdo$. Although we can construct the standard RC algebra $\big(\mfs(\SL2),[\cdot,\cdot]_{\rsdo,\ast}\big)$,  the RC bracket of modular forms defined in \eqref{eq rcb mf} differs from the RC bracket defined by $[\cdot,\cdot]_{\rsdo,\ast}$.  However, one can recover the RC bracket in \eqref{eq rcb mf} using the Serre derivation through the bilinear operator $[\cdot , \cdot]_{\rsdo,\frac{1}{144}E_4,\ast}$, which,  more generally, is defined bellow.

\begin{defi}\label{defi crca}
A \emph{canonical RC algebra} is a pair $(M,[\cdot , \cdot]_{\rstdo,\ew4,\ast})$, where:
\begin{itemize}
   \item $M=\bigoplus_{k\in \zn} M_k$ is a graded algebra of type $\zn$, equipped with a degree-2 derivation $\rstdo: M_\ast \to M_{\ast+2}$ and containing a degree-4 element $\ew4\in M_4$;
   \item $[\cdot , \cdot]_{\rstdo,\ew4,\ast}$ is a family of bilinear operations:
   \[
   [\cdot,\cdot]_{\rstdo,\ew4,n}:M_{k}\otimes M_{l} \to M_{k+l+2n}, \quad \text{for all } k,l\in \zn \text{ and } n\geq 0,
   \]
   defined for any $f\in M_k$ and $g\in M_l$ as:
   \begin{equation}\label{eq crcb}
   [f,g]_{\rstdo,\ew4,n}:=\sum_{j=0}^n(-1)^j\frac{(k+j)_{n-j}(l+n-j)_j}{j!(n-j)!} \rstdo_{(j)}f \rstdo_{(n-j)}g \ ,
   \end{equation}
   where, for all $j\geq 0$ and $h\in M_k$, the operator $\rstdo_{(j)}$ is recursively defined as:
   \begin{equation}\label{eq f_(i)}
   \rstdo_{(0)}h:=h, \quad \rstdo_{(1)}h:=\rstdo h, \quad \rstdo_{(j+1)}h:=\rstdo(\rstdo_{(j)}h) + j(j+k-1)\ew4 \rstdo_{(j-1)}h.
   \end{equation}
\end{itemize}
We refer to $\rstdo$ as the \emph{Serre-type derivation}.
\end{defi}

The canonical RC algebraic structure was first considered by Zagier, who proved in \cite[Proposition~1]{zag94} that when $M$ is a connected graded algebra of type $\Z_{\geq 0}$, any canonical RC algebraic structure on $M$ induces an RC structure. To demonstrate this, he embedded the canonical RC algebra $(M,[\cdot , \cdot]_{\partial,\Lambda,\ast})$ into a standard RC algebra $(\widetilde{M},[\cdot,\cdot]_{\rcdo,\ast})$ for a larger graded algebra $\widetilde{M}$, equipped with a degree-2 derivation $\rcdo$, such that $M\subseteq \widetilde{M}$ and $[\cdot,\cdot]_{\rcdo,\ast}|_{M\otimes M}=[\cdot,\cdot]_{\partial,\Lambda,\ast}$.

In particular, in the classical example of the modular forms on $\SL2$, we observe that $\mfs(\SL2)$ can be embedded in $\qmfs(\SL2)=\mfs(\SL2)[E_2]$, and the bracket structures are related as $[\cdot , \cdot]_{\rcdomf,\ast}|_{\mfs(\SL2)\otimes \mfs(\SL2)}=[\cdot , \cdot]_{\rsdo,\frac{1}{144}E_4,\ast}$.

\begin{theo}\label{thm main1}
  Let $(M,[\cdot , \cdot]_{\partial,\ew4,\ast})$ be a canonical RC algebra. Consider $\widetilde{M}:=M[\qme]=M\otimes_{\k} \k[\qme]$,
where $\qme\notin M$ is an element of degree $2$. Define the degree-$2$ derivation
$\rcdo:\widetilde{M}_\ast\to \widetilde{M}_{\ast+2}$ on the generators of $\widetilde{M}$ by:
\begin{equation}\label{eq der zag}
  \rcdo(f) := \partial(f) + k\qme f \in \widetilde{M}_{k+2}, \quad \text{for any } f\in M_k, \quad \text{and} \quad \rcdo(\qme) := \ew4 + \qme^2 \in \widetilde{M}_4,
\end{equation}
Then, for any $n\geq 0$ and any $f\in M_k$, $g\in M_l$ with $k,l\in\zn$, we have:
\begin{equation}\label{eq []=[]}
  [f,g]_{\partial,\ew4,n} = [f,g]_{\rcdo,n}.
\end{equation}
\end{theo}

Note that this theorem provides a more detailed version of Theorem \ref{thm1}. More precisely, the derivation $\rcdo$ given in \eqref{eq der zag} induces the standard RC algebra $(\widetilde{M},[\cdot,\cdot]_{\rcdo,\ast})$, which contains the canonical RC algebra $(M,[\cdot , \cdot]_{\partial,\ew4,\ast})$ as a sub-RC algebra. Furthermore, an immediate corollary of Theorem \ref{thm main1} establishes that every canonical RC algebra is also an RC algebra.

\begin{cor}
Any canonical RC algebra $(M,[\cdot , \cdot]_{\partial,\ew4,\ast})$ is an RC algebra.
\end{cor}

To prove Theorem \ref{thm main1}, we require additional ingredients, the most significant being the Cohen-Kuznetsov series. In Section \ref{section CC series}, we develop these tools, and in Section \ref{section proof 1}, we present the proof.

\section{Cohen-Kuznetsov series} \label{section CC series}

The Cohen–Kuznetsov series were first introduced independently by Cohen \cite{coh75} and Kuznetsov \cite{kuz75}. Their original definition does not extend to non-positive integers. In this work, we generalize their definition to include non-positive integers as well. A similar definition can be found in \cite{wz22}, published in 2022. However, the author encountered this formulation independently in 2021 during a sabbatical at the Max Planck Institute for Mathematics in Bonn and presented it in a talk at the \emph{Number Theory Lunch Seminar} under the same title.

Note that for any non-negative rational number $k$, the factorial is defined as $k! := \Gamma(k+1)$, where $\Gamma$ denotes the Gamma function. Recall that $\Gamma(k+1) = k\Gamma(k)$, which leads to the property $k! = k \cdot (k-1)!$ for all non-negative rational numbers $k \neq 0$.

Throughout this section, $(\widetilde{M},[\cdot,\cdot]_{\rcdo,\ast})$, $(M,[\cdot , \cdot]_{\partial,\ew4,\ast})$, $\rcdo$, $\rstdo$, $\qme$ and $\ew4$ are considered as in the announcement of Theorem \ref{thm main1}.

\begin{defi}
 Let $f\in M_k$, with $k\in\zn$, and let $X$ be a free parameter. We define the \emph{Cohen--Kuznetsov} (CK) series $CK_\rcdo(f;X)$ and $CK_\rstdo(f;X)$, associated with the  derivations $\rcdo$ and $\rstdo$, respectively, as the following formal power series:
 \begin{description}
  \item {\bf (i)} if  $k\notin \Z_{\leq 0}$, then:
  \begin{align}
   & CK_\rcdo(f;X):=\sum_{n=0}^{+\infty} \frac{\rcdo^{n}f}{n!(k+n-1)!}X^n\,,\label{eq CK_D>0} \\
   & CK_\rstdo(f;X):=\sum_{n=0}^{+\infty} \frac{\rstdo_{(n)}f}{n!(k+n-1)!}X^n\,; \label{eq CK_rsdo>0}
  \end{align}
 \item {\bf (ii)} if $k\in \Z_{\leq 0}$, then:
 \begin{align}
   & CK_\rcdo(f;X):=CK_\rcdo^-(f;X)+CK_\rcdo^+(f;X)\,,\\
   & CK_\rstdo(f;X):=CK_\rstdo^-(f;X)+CK_\rstdo^+(f;X) \,,
  \end{align}
  where
  \begin{align}
   & CK_\rcdo^-(f;X):=\sum_{n=0}^{-k} \left( (-1)^n\frac{(-k-n)!}{n!}\rcdo^{n}f \right) X^n\,, \\
   & CK_\rcdo^+(f;X):=\sum_{n=-k+1}^{+\infty}\left( \frac{1}{n!(k+n-1)!}\rcdo^{n}f\right)X^n\,,  \\
   & CK_\rstdo^-(f;X):=\sum_{n=0}^{-k} \left( (-1)^n\frac{(-k-n)!}{n!}\rstdo_{(n)}f \right) X^n\,, \\
   & CK_\rstdo^+(f;X):=\sum_{n=-k+1}^{+\infty} \left( \frac{1}{n!(k+n-1)!}\rstdo_{(n)}f\right)X^n\,.
  \end{align}
 \end{description}
\end{defi}

The proof of Theorem \ref{thm main1} relies on several technical results, which are presented in the following subsection.

\subsection{Core technical lemmas I}\label{section ctlI}

 \begin{lemm} \label{lemm f_(k)}
 Let $f\in M_k$, with $k\in\zn$. Then, for any $n\in\Z_{\geq 0}$ we have:
 \begin{equation}\label{eq f_(k)}
  \rstdo_{(n)}f=\sum_{j=0}^n (-1)^{n-j}\frac{\ n!(k+j)_{n-j} \ }{(n-j)!j!}\qme^{n-j}\rcdo^jf\,.
 \end{equation}
\end{lemm}
\begin{proof}
 Fix $n\in \Z_{> 0}$ and suppose that \eqref{eq f_(k)} is valid for any non-negative integer $\leq n$ (it is evident that it holds for $n=0,1$).
 Hence, taking into account \eqref{eq der zag} and \eqref{eq f_(k)} for $n$, we obtain:
 \begin{align*}
  \rstdo (\rstdo_{(n)}f)&=\rcdo(\rstdo_{(n)}f)-(k+2n)\qme \rstdo_{(n)}f\\
  &=\sum_{j=0}^n(-1)^{n-j}\frac{n! (k+j)_{n-j}}{(n-j)! j!}
  \left( (n-j)\qme^{n-j-1}(\ew4+\qme^2)\rcdo^jf+\qme^{n-j}\rcdo^{j+1}f \right.\\
  &\qquad\qquad\qquad\qquad\qquad\qquad\qquad\qquad\qquad\qquad\qquad\qquad \left. -(k+2n)\qme^{n-j+1}\rcdo^jf \right)\\
  &=\sum_{j=0}^{n+1}\left\{(-1)^{n+1-j}\frac{n! (k+j)_{n-j}}{(n+1-j)! j!} \Big( -(n-j)(n+1-j)+j(k+j-1) \right.\\
  &\qquad\qquad\qquad\qquad\qquad\qquad\qquad\qquad\qquad\quad   +(k+2n)(n+1-j) \Big) \qme^{n+1-j}\rcdo^jf  \bigg\}\\
  &\ \ \ -n(n+k-1)\ew4 \sum_{j=0}^{n-1} (-1)^{n-1-j}\frac{\ (n-1)!(k+j)_{n-1-j} \ }{(n-1-j)!j!}\qme^{n-1-j}\rcdo^jf\\
  &=\sum_{j=0}^{n+1} (-1)^{n+1-j}\frac{\ (n+1)!(k+j)_{n+1-j} \ }{(n+1-j)!j!}\qme^{n+1-j}\rcdo^jf
  -n(n+k-1)\ew4 \rstdo_{(n-1)}f.
 \end{align*}
On the other hand, by definition given in \eqref{eq f_(i)} we know that
$$\rstdo (\rstdo_{(n)}f)=\rstdo_{(n+1)}f-n(n+k-1)\ew4 \rstdo_{(n-1)}f,$$
which completes the proof of \eqref{eq f_(k)} for $n+1$.
\end{proof}

\begin{rem} \label{rem f_(n)}
 If $f\in M_k$, with $k\in\zn$, then for all $n\in\Z_{\geq 0}$ we have:
 \[
  \rstdo_{(n)}f=\left \{ \begin{array}{l}
\sum_{j=0}^n (-1)^{n-j}\frac{\ n!(k+n-1)! \ }{j!(n-j)!(k+j-1)!}\qme^{n-j}\rcdo^jf; \qquad\, {\rm if} \ k\notin \Z_{\leq 0}, \\\\
\sum_{j=-k+1}^n (-1)^{n-j}\frac{\ n!(k+n-1)! \ }{j!(n-j)!(k+j-1)!}\qme^{n-j}\rcdo^jf; \ \ {\rm if} \ k\in \Z_{\leq 0} \ {\rm and} \ n\geq -k+1, \\\\
\sum_{j=0}^n \frac{\ n!(-k-j)! \ }{j!(n-j)!(-k-n)!}\qme^{n-j}\rcdo^jf; \qquad\qquad\qquad {\rm if} \ k\in \Z_{\leq 0} \ {\rm and} \ n< -k+1.
\end{array} \right.
 \]
\end{rem}

In the following, for any $f\in \widetilde{M}$, by $e^{fX}$ we mean the formal power series $\sum_{j=0}^\infty \frac{f^j}{j!}X^j$.

\begin{lemm} \label{lemm exp(qme)CK}
 Let $f\in M_k$, with $k\in\zn$.
  \begin{enumerate}
  \item If $k\notin \Z_{\leq 0}$, then
  \begin{equation}
    e^{-\qme X} CK_\rcdo(f;X)=CK_\rstdo(f;X)\,.
  \end{equation}
  \item If $k\in \Z_{\leq 0}$, then
  \begin{align}
     &e^{-\qme X} CK_\rcdo^-(f;X)=CK_\rstdo^-(f;X) +\sum_{n=-k+1}^{+\infty}A_{n}(\qme,f)X^n, \label{eq CK-r<0}\\
     &e^{-\qme X} CK_\rcdo^+(f;X)=CK_\rstdo^+(f;X), \label{eq eCK+}
  \end{align}
where
\begin{equation} \label{eq A_n}
A_{n}(\qme,f):=(-1)^n \sum_{j=0}^{-k}
    \frac{(-k-j)!}{j!(n-j)!}\qme^{n-j}\rcdo^jf  .
\end{equation}
 \end{enumerate}
\end{lemm}
\begin{proof}
 We prove \eqref{eq CK-r<0}, and the other cases can be proved in a similar manner.  To this end, we assume that $k\in \Z_{\leq 0}$, which gives:
  \begin{align*}
  e^{-\qme X} CK_\rcdo^-(f;X)&=\sum_{j=0}^\infty (-1)^j\frac{\qme^j}{j!}X^j\cdot \sum_{n=0}^{-k}  (-1)^n\frac{(-k-n)!}{n!}\rcdo^nf  X^n \\
  &=\sum_{n=0}^{-k} \left( \sum_{j=0}^{n}(-1)^n\frac{(-k-j)!}{j!(n-j)!}\qme^{n-j}\rcdo^jf  \right)X^n \\
  &\qquad\qquad \qquad\ +  \sum_{n=-k+1}^{+\infty} \left( \sum_{j=0}^{-k}(-1)^n\frac{(-k-j)!}{j!(n-j)!}\qme^{n-j}\rcdo^jf  \right) X^n \\
  &= CK_\rstdo^-(f;X)+\sum_{n=-k+1}^{+\infty} (-1)^n\left( \sum_{j=0}^{-k}
    \frac{(-k-j)!}{j!(n-j)!}\qme^{n-j}\rcdo^jf  \right) X^n\,.
 \end{align*}
The last equality follows immediately from Remark \ref{rem f_(n)}.
\end{proof}

\begin{rem} \label{rem []_rcdo,n}
  Let $f\in M_k$ and $g\in M_l$, with $k,l\in \Z_{\leq 0}$ and $l\leq k <0$.
  \begin{itemize}
    \item  If $0\leq n\leq -k$, then:
    \[
  [f,g]_{\rcdo,n}=\sum_{j=0}^n(-1)^j\frac{(k+j)_{n-j}(l+n-j)_j}{j!(n-j)!} \rcdo^{j}f \rcdo^{n-j}g   \, .
 \]
    \item  If $-k+1\leq n\leq -l$, then:
    \[
  [f,g]_{\rcdo,n}=\sum_{j=-k+1}^n(-1)^j\frac{(k+j)_{n-j}(l+n-j)_j}{j!(n-j)!} \rcdo^{j}f \rcdo^{n-j}g   \, .
 \]
    \item  If $-l+1\leq n\leq -k-l+1$, then:
    \[
  [f,g]_{\rcdo,n}=0   \, .
 \]
    \item  If $-k-l+2\leq n$, then:
    \[
  [f,g]_{\rcdo,n}=\sum_{j=-k+1}^{n+l-1}(-1)^j\frac{(k+j)_{n-j}(l+n-j)_j}{j!(n-j)!} \rcdo^{j}f \rcdo^{n-j}g   \, .
 \]
  \end{itemize}
We can find analogous statements for $[f,g]_{\rstdo,\ew4,n}$.
\end{rem}

\section{Proof of Theorem \ref{thm main1}} \label{section proof 1}
We prove the theorem for the case where $k,l\in \Z_{<0}$, and the remaining cases will be proved similarly. Without loss of generality, we can assume that $l\leq k<0$. Due to Remark \ref{rem []_rcdo,n}, we analyze the following 4 possibilities individually.
\begin{description}
  \item[1) ($0\leq n\leq -k$).]  We first observe that:
\begin{align}\label{eq rcdo CK-CK-}
  &CK_\rcdo^-(f;-X)CK_\rcdo^-(g;X)=\sum_{n=0}^{-k} (-k-n)!(-l-n)![f,g]_{\rcdo,n}X^n \\
  &\qquad\qquad\qquad\qquad\qquad\qquad\qquad\qquad \qquad\quad +\sum_{n=-k+1}^{-k-l}C_n(f,g)X^n,\nonumber\\
  &CK_\rstdo^-(f;-X)CK_\rstdo^-(g;X)=\sum_{n=0}^{-k} (-k-n)!(-l-n)![f,g]_{\rstdo,\ew4,n}X^n \label{eq rstdo CK-CK-}\\
  &\qquad\qquad\qquad\qquad\qquad\qquad\qquad\qquad \qquad\quad +\sum_{n=-k+1}^{-k-l}D_n(f,g)X^n,\nonumber
\end{align}
where
\[
  C_n(f,g):=\left \{ \begin{array}{l}
\sum_{j=0}^{-k}(-1)^{n+j}\frac{(-k-j)! (-l-n+j)!}{j! (n-j)!}\rcdo^jf \rcdo^{n-j}g; \quad\  {\rm if} \ -k+1\leq n \leq -l, \\\\
\sum_{j=n+l}^{-k}(-1)^{n+j}\frac{(-k-j)! (-l-n+j)!}{j! (n-j)!}\rcdo^jf \rcdo^{n-j}g; \  {\rm if} \ -l+1 \leq n \leq -k-l,
\end{array} \right.
\]
and $D_n(f,g)$ is analogous to $C_n(f,g)$, with $\rcdo^j$ replaced by $\rstdo_{(j)}$, for all $j$.
On the other hand, by virtue of Lemma \ref{lemm exp(qme)CK} and \eqref{eq rstdo CK-CK-}, we have:
  \begin{align}
      e^{\qme X}CK_\rcdo^-(f;-X) e^{-\qme X}&CK_\rcdo^-(g;X) =
    \sum_{n=0}^{-k} (-k-n)!(-l-n)![f,g]_{\rstdo,\ew4,n}X^n \label{eq eCK-eCK-}\\
&+\sum_{n=-k+1}^{-k-l}D_n(f,g)X^n+\sum_{n=-k+1}^{\infty} E_n(f,g)X^n,\nonumber
  \end{align}
where
\begin{align*}
\sum_{n=-k+1}^{\infty} E_n(f,g)X^n:&=CK_\rstdo^-(f;-X)\cdot \sum_{n=-l+1}^{\infty} A_n(\qme,g)X^n\\
                                   &+CK_\rstdo^-(g;X)\cdot \sum_{n=-k+1}^{\infty} (-1)^nA_n(\qme,f)X^n\\
                                   &+\sum_{n=-k+1}^{\infty} (-1)^nA_n(\qme,f)X^n \cdot \sum_{n=-l+1}^{\infty} A_n(\qme,g)X^n.
\end{align*}
Therefore, the equality \eqref{eq rcdo CK-CK-}=\eqref{eq eCK-eCK-} yields that $[f,g]_{\rstdo,\ew4,n}=[f,g]_{\rcdo,n}\,$, for all $0\leq n\leq -k$.
  \item[2) ($-k+1\leq n\leq -l$).] We obtain that:
  { \begin{align}
    CK_\rcdo^+(f;-X)&CK_\rcdo^-(g;X) =\sum_{n=-k+1}^{-l}(-1)^n \frac{(-l-n)!}{(k+n-1)!}[f,g]_{\rcdo,n}X^n \label{eq CK+CK-}\\
&+\sum_{n=-l+1}^{\infty}\left( \sum_{j=n+l}^{n} (-1)^n \frac{(-n-l+j)!}{j! (n-j)!(k+j-1)!}\rcdo^j f\rcdo^{n-j}g \right) X^n. \nonumber
\end{align}
By using Lemma \ref{lemm exp(qme)CK}, we get:
\begin{align}
    e^{\qme X}CK_\rcdo^+&(f;-X) e^{-\qme X}CK_\rcdo^+(g;X) =
    \sum_{n=-k+1}^{-l} (-1)^n\frac{(-l-n)!}{(n+k-1)!}[f,g]_{\rstdo,\ew4,n}X^n \label{eq eCK+eCK-}\\
&+\sum_{n=-l+1}^{\infty}\left( \sum_{j=n+l}^{n} (-1)^n \frac{(-n-l+j)!}{j! (n-j)!(k+j-1)!}\rstdo_{(j)} f\rstdo_{(n-j)}g \right) X^n \nonumber \\
&+CK_\rstdo^+(f;-X)\cdot \sum_{n=-l+1}^{\infty} A_n(\qme,g)X^n. \nonumber
  \end{align}}
  By comparing the coefficients of $X^n$ in the both sides of the equality \eqref{eq CK+CK-}=\eqref{eq eCK+eCK-}, we show that $[f,g]_{\rstdo,\ew4,n}=[f,g]_{\rcdo,n}$, for all $-k+1\leq n\leq -l$.
  \item[3) ($-l+1\leq n\leq -k-l+1$).] When $-l+1\leq n\leq -k-l+1$, it is straightforward to observe that
   $[f,g]_{\rstdo,\ew4,n}=0=[f,g]_{\rcdo,n}$.
  \item[4) ($-k-l+2\leq n$).] To establish the result in this case, we first obtain:
   \begin{align}
    CK_{\rcdo}^+(f;-X)CK_\rcdo^+(g;X) &=\sum_{n=-k-l+2}^{\infty} \frac{[f,g]_{\rcdo,n}}{(k+n-1)!(l+n-1)!}X^n. \label{eq CK+CK+}
  \end{align}
By Lemma \ref{lemm exp(qme)CK}, we have:
   \begin{align}
    e^{\qme X}CK_\rcdo^+(f;-X) e^{-\qme X}&CK_\rcdo^+(g;X) = CK_\rstdo^+(f;-X) CK_\rstdo^+(g;X) \label{eq eCK+eCK+} \\
    &=\sum_{n=-k-l+2}^{\infty} \frac{[f,g]_{\rstdo,\ew4,n}}{(k+n-1)!(l+n-1)!}X^n.\nonumber
  \end{align}
  Finally, equating \eqref{eq CK+CK+} and \eqref{eq eCK+eCK+}, we conclude that $[f,g]_{\rstdo,\ew4,n}=[f,g]_{\rcdo,n}$, for all $n\geq -k-l+2$.
\end{description}

\section{Core technical lemmas II} \label{section ctlII}
Throughout this section, again, $(\widetilde{M},[\cdot,\cdot]_{\rcdo,\ast})$, $(M,[\cdot , \cdot]_{\partial,\ew4,\ast})$, $\rcdo$, $\rstdo$, $\qme$ and $\ew4$ are considered as in the announcement of Theorem \ref{thm main1}.
To prove Theorem \ref{thm2}, we require some technical results concerning $\qme$, which are stated in this section.
\begin{lemm} \label{lemm Psi}
Define
\begin{align}
  & \rstdo_{(0)}\qme:=\qme, \\
  & \rstdo_{(1)}\qme=\rstdo \qme:=\ew4-\qme^2, \\
  & \rstdo_{(j+1)}\qme=\rstdo(\rstdo_{(j)}\qme)+j(j+1)\ew4\rstdo_{(j-1)}\qme, \ j=1,2,3,\ldots.
\end{align}
Then, for any non-negative integer $n$, we obtain:
\begin{equation}\label{eq qme_(k)}
  \rstdo_{(n)}\qme=\Psi_n(\ew4,\rstdo\ew4,\ldots,\rstdo^{n-1}\ew4)+(-1)^n n!\qme^{n+1},
\end{equation}
in which $\Psi_n=\Psi_n(\ew4,\rstdo\ew4,\ldots,\rstdo^{n-1}\ew4)$ is a polynomial in  $\ew4,\rstdo\ew4,\ldots,\rstdo^{n-1}\ew4$
given as follows:
\begin{align}
  & \Psi_0=0, \\
  & \Psi_1=\ew4, \\
  & \Psi_{j+1}=\rstdo \Psi_{j}+j(j+1)\ew4\Psi_{j-1}, \ j=1,2,3,\ldots.
\end{align}
In particular, $\Psi_{n}=\rstdo_{(n)}\qme+(-1)^{n+1} n!\qme^{n+1} \in M_{2n+2}$.
\end{lemm}
\begin{proof}
 Using induction, the proof becomes straightforward.
\end{proof}

\begin{rem}
Analogous to Lemma \ref{lemm f_(k)} we can see that for any $n\in\Z_{\geq 0}$:
 \begin{align}\label{eq qme_(k)2}
  \rstdo_{(n)}\qme&=\sum_{j=0}^n (-1)^{n-j}\frac{\ n!(j+2)_{n-j} \ }{j!(n-j)!}\qme^{n-j}\rcdo^j\qme\\
  &=\sum_{j=0}^n (-1)^{n-j}\frac{\ n!(n+1)! \ }{j!(j+1)!(n-j)!}\qme^{n-j}\rcdo^j\qme, \nonumber
 \end{align}
 and after using \eqref{eq qme_(k)} we obtain:
\begin{equation} \label{eq qme^(k)}
 \rcdo^n\qme=\Psi_{n}+(-1)^{n}n!\qme^{n+1}
 +\sum_{j=0}^{n-1}(-1)^{n+1-j}\frac{n!(n+1)!}{j!(j+1)!(n-j)!}\qme^{n-j}\rcdo^j\qme \,.
\end{equation}
 In the same way as Lemma \ref{lemm exp(qme)CK} we observe that:
 \begin{equation}
    e^{-\qme X} CK_\rcdo(\qme;X)=CK_\rstdo(\qme;X)\,,
  \end{equation}
  in which $CK_\rcdo(\qme;X)$ and $CK_\rstdo(\qme;X)$ are defined similar to \eqref{eq CK_D>0} and \eqref{eq CK_rsdo>0}, respectively.
  Finally, analogous to Theorem \ref{thm main1} we can demonstrate that for any $f\in M_k$, with  $k\in\zn$:
\begin{equation}\label{eq [qme,f]=[qme,f]}
  [\qme,f]_{\rcdo,n}=[\qme,f]_{\partial,\ew4,n} \ .
\end{equation}
\begin{equation}\label{eq [qme,qme]=[qme,qme]}
  [\qme,\qme]_{\rcdo,n}=[\qme,\qme]_{\partial,\ew4,n} \ .
\end{equation}
where $[\qme,f]_{\partial,\ew4,n}$ and $[\qme,\qme]_{\partial,\ew4,n}$ are defined similar to \eqref{eq crcb}.
\end{rem}

\begin{lemm} \label{lemm qme^(k+1)}
 For any non-negative integer $n$ we have:
 \begin{equation} \label{eq qme^(k+1)}
  \rcdo^{n+1}\qme=\sum_{j=0}^{n+1}\frac{(n+1)!(n+2)!}{j!(j+1)!(n+1-j)!}\qme^{n+1-j}\Psi_j+(n+1)!\qme^{n+2}\,.
 \end{equation}
\end{lemm}
\begin{proof}
 First, we prove that for any integer $0\leq m\leq n+1$, the following holds:
 \begin{align} \label{eq qme^(k+1)s}
  \rcdo^{n+1}\qme&=\sum_{j=n+1-m}^{n+1}\frac{(n+1)!(n+2)!}{j!(j+1)!(n+1-j)!}\qme^{n+1-j}\Psi_j \\
  &+\qme^{n+2}\sum_{j=n+1-m}^{n+1}(-1)^{j}\frac{(n+1)!(n+2)!}{(j+1)!(n+1-j)!} \nonumber\\
  &+\sum_{j=0}^{n-m}(-1)^{n+2-j-m}\frac{(n+1)!(n+2)!\prod_{i=0}^{m-1}(n-j-i)}{m!j!(j+1)!(n+1-j)!}\qme^{n+1-j}\rcdo^{j}\qme \nonumber\,,
 \end{align}
 where we assume that $\sum_{j=0}^{-1}\Box_j=0$ and $\prod_{j=0}^{-1}\Box_j=1$.
Setting $m=0$ and using \eqref{eq qme^(k)} for $n+1$, we immediately obtain \eqref{eq qme^(k+1)s} for this base case.
Next, assuming that \eqref{eq qme^(k+1)s} holds for some $0\leq m\leq n$, we prove that
it also holds for $m+1$. To this end, we compute $\rcdo^{n-m}\qme$ from \eqref{eq qme^(k)} and substitute it into \eqref{eq qme^(k+1)s} for $m$. This yields:
 \begin{align} \label{}
  \rcdo^{n+1}\qme&=\sum_{j=n+1-m}^{n+1}\frac{(n+1)!(n+2)!}{j!(j+1)!(n+1-j)!}\qme^{n+1-j}\Psi_j \nonumber\\
  &+\qme^{n+2}\sum_{j=n+1-m}^{n+1}(-1)^{j}\frac{(n+1)!(n+2)!}{(j+1)!(n+1-j)!} \nonumber\\
  &+\sum_{j=0}^{n-m-1}(-1)^{n+2-j-m}\frac{(n+1)!(n+2)!\prod_{i=0}^{m-1}(n-j-i)}{m!j!(j+1)!(n+1-j)!}\qme^{n+1-j}\rcdo^{j}\qme \nonumber \\
  &+\frac{(n+1)!(n+2)!}{(n-m)!(n-m+1)!(m+1)!}\qme^{m+1}\Psi_{n-m}\nonumber\\
  &+(-1)^{n-m}\frac{(n+1)!(n+2)!}{(n-m+1)!(m+1)!}\qme^{n+2} \nonumber \\
  &+\sum_{j=0}^{n-m-1}(-1)^{n-m+1-j}\frac{(n+1)!(n+2)!}{j!(j+1)!(n-m-j)!(m+1)!}\qme^{n+1-j}\rcdo^{j}\qme \nonumber \\
  &=\sum_{j=n+1-(m+1)}^{n+1}\frac{(n+1)!(n+2)!}{j!(j+1)!(n+1-j)!}\qme^{n+1-j}\Psi_j \nonumber \\
  &+\qme^{n+2}\sum_{j=n+1-(m+1)}^{n+1}(-1)^{j}\frac{(n+1)!(n+2)!}{(j+1)!(n+1-j)!} \nonumber\\
  &+\sum_{j=0}^{n-(m+1)}(-1)^{n+2-j-(m+1)}\frac{(n+1)!(n+2)!\prod_{i=0}^{(m+1)-1}(n-j-i)}{(m+1)!j!(j+1)!(n+1-j)!}\qme^{n+1-j}\rcdo^{j}\qme \nonumber\,,
 \end{align}
 which establishes \eqref{eq qme^(k+1)s} for $m+1$. Finally, setting $m=n+1$ in \eqref{eq qme^(k+1)s} and noting that
 $$\sum_{j=0}^{n+1}(-1)^{j}\frac{(n+1)!(n+2)!}{(j+1)!(n+1-j)!}=(n+1)!,$$
we obtain \eqref{eq qme^(k+1)}, completing the proof.
\end{proof}

\section{Proof of Theorem \ref{thm2}} \label{section proof 2}

We first prove that $\KK^{(n)}f\in M_{k+2+2n}$. Using equation \eqref{eq [qme,f]=[qme,f]} and then applying \eqref{eq qme_(k)}, we obtain:
\begin{align}
 \KK^{(n)}f=&\rcdo^{n+1}f-(n+k)[\qme,f]_{\rcdo,n}=\rcdo^{n+1}f-(n+k)[\qme,f]_{\partial,\ew4,n} \nonumber \\
   =&\rcdo^{n+1}f-(n+k)\sum_{j=0}^n(-1)^j\binom{n+1}{n-j}\binom{n+k-1}{j}\rstdo_{(j)}\qme \rstdo_{(n-j)}f  \nonumber\\
  =&\rcdo^{n+1}f-(n+k)\sum_{j=0}^n(-1)^j\binom{n+1}{n-j}\binom{n+k-1}{j}\Psi_{j}\rstdo_{(n-j)}f\nonumber\\
  & -(n+k)\sum_{j=0}^n(-1)^{2j}j!\binom{n+1}{n-j}\binom{n+k-1}{j}\qme^{j+1}\rstdo_{(n-j)}f   \nonumber\\
   =&\rcdo^{n+1}f-(n+k)\sum_{j=0}^n(-1)^j\binom{n+1}{n-j}\binom{n+k-1}{j}\Psi_{j}\rstdo_{(n-j)}f\nonumber\\
  & -(n+k)\sum_{j=0}^n\frac{1}{j!}\bigg( \sum_{i=j}^n (-1)^{i-j}\frac{(n+2-i)_i(k+i)_{n-i}(k+j)_{i-j}}{(i-j)!} \bigg)
  \qme^{n+1-j}\rcdo^{j}f \nonumber \\
  =&-(n+k)\sum_{j=0}^n(-1)^j\binom{n+1}{n-j}\binom{n+k-1}{j}\Psi_{j}\rstdo_{(n-j)}f+\rstdo_{(n+1)}f . \label{eq last eq}
\end{align}
In the above equalities, we substitute $\rstdo_{(n-j)}f$ and $\rstdo_{(n+1)}f$ using equation \eqref{eq f_(k)}. Additionally, for
integers $0\leq j\leq n$, we use the following identity:
\[
 \sum_{i=j}^n \frac{(-1)^{i-j}}{(n+1-i)!(i-j)!}=\sum_{i=0}^{n-j} \frac{(-1)^{i}}{(n-j+1-i)!\,i!}=\frac{(-1)^{n-j}}{(n-j+1)!}\,.
\]
Since all terms in equation \eqref{eq last eq} belong to $M_{k+2+2n}$, this completes the proof that $\KK^{(n)}f\in M_{k+2+2n}$.

To prove that $\KK^{(n)}\qme\in M_{4+2n}$, we first note that if $n$ is odd, then there is nothing to prove since $\KK^{(n)}\qme=0$. Hence, we assume that $n$ is an even positive integer. Using equations \eqref{eq [qme,qme]=[qme,qme]}, \eqref{eq qme_(k)} and \eqref{eq qme^(k+1)}, along with the identity
\[
\sum_{j=0}^{n+1}(-1)^{j}\frac{(n+1)!(n+2)!}{(j+1)!(n+1-j)!}=(n+1)!,
\]
we obtain:
\begin{align}
 \KK^{(n)}\qme=&-(n+2)[\qme,\qme]_{\rcdo,n}+2\rcdo^{n+1}\qme=-(n+2)[\qme,\qme]_{\partial,\ew4,n}+2\rcdo^{n+1}\qme \nonumber \\
  =& -(n+2)\sum_{j=0}^n \left\{ (-1)^j\binom{n+1}{n-j}\binom{n+1}{j} \Big(\Psi_j+(-1)^{j} j! \qme^{j+1} \Big) \right. \nonumber \\
 &\quad\qquad\qquad\qquad\qquad\qquad  \left( \Psi_{n-j}+(-1)^{n-j} (n-j)! \qme^{n+1-j} \right) \bigg\}  +2\rcdo^{n+1}\qme \nonumber\\
   =&- \sum_{j=0}^n(-1)^{j}\frac{(n+1)!(n+2)!}{j!(j+1)!(n-j)!(n+1-j)!}\Psi_{j}\Psi_{n-j} \nonumber\\
   &- 2\sum_{j=0}^n\frac{(n+1)!(n+2)!}{j!(j+1)!(n+1-j)!}\qme^{n+1-j}\Psi_{j}- 2(n+1)!\qme^{n+2}+2\rcdo^{n+1}\qme \nonumber\\
  =&-(n+2)\sum_{j=0}^n(-1)^{j}\binom{n+1}{n-j}\binom{n+1}{j} \Psi_{j}\Psi_{n-j}+2\Psi_{n+1}. \nonumber
\end{align}
Every term in the final equation lies in $M_{4+2n} $, confirming that $\KK^{(n)}\qme\in M_{4+2n}$.


\section{Lie algebra $\sl2$ and RRC systems} \label{section sl2}

Consider the following system of nonlinear ODEs, known as the Ramanujan system:
\begin{equation} \label{eq ramanujan}
 \left \{ \begin{array}{l}
t_1'=\frac{1}{12}(t_1^2-t_2) \\
t_2'=\frac{1}{3}(t_1t_2-t_3) \\
t_3'=\frac{1}{2}(t_1t_3-t_2^2)
\end{array} \right., \qquad with  \  \ast'=q\frac{\partial \ast}{\partial q}=\frac{1}{2\pi i}\frac{d}{d \tau}\ \ and \ \ q=e^{2\pi i\tau}\ ,
\end{equation}
Ramanujan, in \cite{ra16}, showed that the triple $(E_2,E_4,E_6)$ of Eisenstein series forms a solution to this system. This, together with the fact that $\qmfs(\SL2)=\C[E_2,E_4,E_6]$, implies that the usual normalized derivation on $\qmfs(\SL2)$ can be written as \eqref{eq qmf der}.
 Moreover, if we consider the weight derivation $\rdo=2E_2\frac{\partial }{\partial E_2}+4E_4\frac{\partial }{\partial E_4}+6E_6\frac{\partial }{\partial E_6}$ and set $\cdo=-12\frac{\partial }{\partial E_2}$, then their Lie brackets satisfy the identities:
\[
[\rcdomf,\cdo]=\rdo \ , \ \ [\rdo,\rcdomf]=2\rcdomf \ , \ \ [\rdo,\cdo]=-2\cdo.
\]
In other words, $\rcdomf, \rdo$ and $\cdo$ endow $\qmfs(\SL2)$ with an $\sl2$-module structure.
Furthermore, if $\Gamma\subset {\rm PSL}_2(\R)$ is a non-cocompact Fuchsian subgroup, then Zagier in \cite{bhgz} showed that there exists a weight-$2$ quasi-modular form $\qme$ (which is not a modular form) such that $\qmfs(\Gamma)=\mfs(\Gamma)[\qme]$. Additionally, he proved the existence of a derivation $\cdo$, which is a constant multiple of $\frac{\partial}{\partial \qme}$, on $\qmfs(\Gamma)$ such that together with the usual derivation $\rcdomf$ and the weight derivation $\rdo$, it provides $\qmfs(\Gamma)$ with an $\sl2$-module structure. By virtue of Theorem \ref{thm4}, this is equivalent to the fact that the RC algebra  $\big(\mfs(\Gamma),[\cdot,\cdot]_\ast\big)$ is canonical.

If we substitute $t_1$ by $\frac{1}{12}t_1$, then the Ramanujan system \eqref{eq ramanujan} transforms into the following form::
\begin{equation} \label{eq ramanujan2}
 \left \{ \begin{array}{l}
t_1'=t_1^2-\frac{1}{144}t_2 \\
t_2'=4t_1t_2-\frac{1}{3}t_3 \\
t_3'=6t_1t_3-\frac{1}{2}t_2^2
\end{array} \right..
\end{equation}
Here, the weights 2 and 4 of $E_4$ and $E_6$, respectively, appear as multiples of $t_1t_2$ and $t_1t_3$ in the system. In fact, any finitely generated canonical RC algebra can be associated with such a system, and vice versa. We demonstrated this in \cite[Theorem 2.1]{bn23}, and due to the importance of this theorem for the present work, we restate it here without proof. This result is also a key tool in the proof of Theorem \ref{thm4}.
\begin{theo} \label{thm5} {\bf (\cite[Theorem 2.1]{bn23})}
Let $\rcdo$ be a derivation on a graded algebra of type $\zn$ over a field $\k$ of characteristic zero. Consider the following system of nonlinar ODEs
\begin{equation} \label{eq rrc system}
\left\{
  \begin{array}{ll}
    \rcdo{\vs_1}=\vs_1^2+\prs_1(\vs_2,\vs_3,\ldots,\vs_\nvs)  \\
    \rcdo{\vs_2}={\wss_2}\vs_1\vs_2+\prs_2(\vs_2,\vs_3,\ldots,\vs_\nvs)  \\
    \rcdo{\vs_3}={\wss_3}\vs_1\vs_3+\prs_3(\vs_2,\vs_3,\ldots,\vs_\nvs)  \\
    \vdots  \\
    \rcdo{\vs_\nvs}={\wss_\nvs}\vs_1\vs_\nvs+\prs_\nvs(\vs_2,\vs_3,\ldots,\vs_\nvs)
  \end{array}
\right.\,,
\end{equation}
where~$\vs_j$ is of degree $\wss_j\in\zn\,,(\wss_1=2)$, and $\prs_j(\vs_2,\vs_3,\ldots,\vs_\nvs)\in \k[\vs_2,\dots,\vs_d]$ is a quasi-homogeneous polynomial of degree~$\wss_j+2$.
\begin{enumerate}
\item The algebra $\k[\vs_1,\dots,\vs_d]$ is a standard RC algebra with derivation $\rcdo$, and the sub-algebra $\k[\vs_2,\dots,\vs_d]$ is a canonical RC algebra.
\item Conversely, every finitely generated canonical RC algebra arises in this way.
\end{enumerate}
\end{theo}

We refer to the system \eqref{eq rrc system} as the \emph{Ramanujan system of Rankin-Cohen type} (RRC system). Theorem \ref{thm main1} plays a crucial role in the proof of Theorem \ref{thm5}, and the reader is referred to the cited reference for its proof. It is worth noting that in the proof, we consider $\qme=\vs_1$, $\ew4=\prs_1(\vs_2,\vs_3,\ldots,\vs_\nvs)$ and $\rstdo \vs_j=\prs_j(\vs_2,\vs_3,\ldots,\vs_\nvs)$ for $j=2,3,\ldots, \nvs$. Furthermore, the derivation $\rcdo$ defined by the RRC system \eqref{eq rrc system}, together with the weight derivation $\rdo=\sum_{j=1}^{\nvs}\wss_jt_j\frac{\partial}{\partial t_j}$ and the derivation $\cdo=-\frac{\partial}{\partial t_1}$, endows $\k[\vs_1,\vs_2,\ldots,\vs_d]$ with an $\sl2$-module structure.

\begin{exam} \label{ex cyqms}({\bf A noteworthy non-classical example})
The author in \cite{younes3} studied the space of CY (quasi-)modular forms ($\cyqmfs$) $\cymfs$ associated with a certain moduli space $\T=\T_n$ of a family of CY $n$-folds arising from the Dwork family, with $n\in \N$, whose underlying graded algebra is of type $\Z$. In particular, for any $n$, he identified the derivations $\rcdo(={\sf R})$, $\rdo(={\sf H})$ and $\cdo(={\sf F})$, which equip $\cyqmfs$ with an $\sl2$-module structure. For $n=1$ and $n=2$,  we recover the classical space of (quasi-)modular forms for $\Gamma_0(3)$ and $\Gamma_0(2)$, respectively, which have been studied in detail in \cite{younes2}. For $n\geq 3$, we obtain non-classical examples for standard and canonical RC algebras, as well as RRC systems. We now present the case $n=3$, which holds significant importance for both mathematicians and physicists, as it is associated with mirror quintic 3-folds. In this case, we observe that $\dim \T=7$ and $\cyqmfs=\C[\t_1,\t_2,\ldots,\t_7,\t_8]$, where $\t_8:=\frac{1}{\t_5(\t_1^5-\t_5)}$ and the $\t_j$'s, with $1\leq j \leq 7$, are solutions of a vector field known as modular vector field on $\T$ (see \cite{movnik,younes3} for more details). Additionally, the degrees of the generators are given by $\deg(\t_1)=1, \ \deg(\t_2)=2, \ \deg(\t_3)=3, \ \deg(\t_4)=0, \ \deg(\t_5)=5, \ \deg(\t_6)=1, \ \deg(\t_7)=2$ and $\deg(\t_8)=-10$. We find in \cite[Example 5.1]{younes3} that:
\begin{align}
\rcdo&=(\t_3-\t_1\t_2)\frac{\partial}{\partial
\t_1}+\left(5^{-4}\t_3^3\t_4\t_5\t_8-\t_2^2\right)\frac{\partial}{\partial \t_2} \label{eq mvf R3}
+\left(5^{-4}\t_3^3\t_5\t_6\t_8-3\t_2\t_3\right)\frac{\partial}{\partial
\t_3}\\&  +(-\t_2\t_4-\t_7)\frac{\partial}{\partial \t_4}
+(-5\t_2\t_5)\frac{\partial}{\partial
\t_5}+(5^5\t_1^3-\t_2\t_6-2\t_3\t_4)\frac{\partial}{\partial
\t_6}\nonumber\\&+(-5^4\t_1\t_3-\t_2\t_7)\frac{\partial}{\partial \t_7}+(-5\t_1^4\t_3\t_5\t_8^2+10\t_2\t_8)\frac{\partial}{\partial \t_8}\,, \nonumber\\
\rdo&=\t_1\frac{\partial}{\partial \t_1}+2\t_2\frac{\partial}{\partial
\t_2}+3\t_3\frac{\partial}{\partial \t_3}+5\t_5\frac{\partial}{\partial
\t_5}+\t_6\frac{\partial}{\partial \t_6}+2\t_7\frac{\partial}{\partial \t_7}-10\t_8\frac{\partial}{\partial \t_8}\,,\label{eq rdo T3} \\
\cdo&=\frac{\partial}{\partial \t_2}-\t_4\frac{\partial}{\partial
\t_7}\,.
\end{align}
We observe that $\rcdo$ is a degree-$2$ derivation on $\cyqmfs$, which endows it with a standard RC algebra structure. Moreover, we obtain:
\[
[\rcdo,\cdo]=\rdo, \ \ [\rdo,\rcdo]=2\rcdo, \ \ [\rdo,\cdo]=-2\cdo,
\]
$\cdo \t_j=0$ for $j=1,3,4,5,6,8$, and $\cdo(\t_7+\t_2\t_4)=0$. Thus, setting $\tilde{\t}_7:=\t_7+\t_2\t_4$ and $\cymfs:=\ker \cdo=\C[\t_1,\t_3,\t_4,\t_5,\t_6,\tilde{\t}_7,\t_8]$, we obtain $\cyqmfs=\cymfs[\t_2]$. Consequently, by Theorem \ref{thm4}, the space of CY modular forms $\cymfs$ is a canonical RC algebra. However, if we represent $\rcdo$ as a system of ODEs, it is not match the form of an RRC system introduced in \eqref{eq rrc system}. Nonetheless, by substituting the generators $\t_2$ and $\t_7$ by $\tilde{\t}_2:=-\t_2$ and $\tilde{\t}_7$, respectively, and--by abuse of notation--denoting them again by $\t_2$ and $\t_7$, we obtain the desired RRC system:
\begin{equation} \label{eq rrc system T3}
\rcdo: \ \left\{
  \begin{array}{ll}
    \rcdo{\t_1}=\t_1\t_2+\t_3  \\
    \rcdo{\t_2}= \t_2^2-5^{-4}\t_3^3\t_4\t_5\t_8 \\
    \rcdo{\t_3}= 3\t_2\t_3+5^{-4}\t_3^3\t_5\t_6\t_8 \\
    \rcdo{\t_4}= -\t_7 \\
    \rcdo{\t_5}= 5\t_2\t_5  \\
    \rcdo{\t_6}= \t_2\t_6+5^5\t_1^3-2\t_3\t_4 \\
    \rcdo{\t_7}= 2\t_2\t_7-5^4t_1t_3+5^{-4}\t_3^3\t_4^2\t_5\t_8 \\
    \rcdo{\t_8}= -10\t_2\t_8-5\t_1^4\t_3\t_5\t_8^2
  \end{array}
\right.\,.
\end{equation}
Furthermore, the weight derivation $\rdo$ remains the same as in \eqref{eq rdo T3} and $\cdo=-\frac{\partial}{\partial \t_2}$. It is worth to mention that the importance of $\t_j$'s becomes even more evident when we consider that Yukawa coupling, partition functions of topological string theory, and mirror map can all be expressed in terms of $\t_j$'s (see \cite{younes3} and references therein).
\end{exam}

Note that in the above example, $\cyqmfs$ contains elements of negative degree, such as $\t_8$, as well as non-constant elements of degree $0$, such as $\t_4$. The stated results in this example hold thanks to Theorem \ref{thm1}.

Any finitely generated RC algebra $(M,[\cdot,\cdot]_\ast)$ that contains a homogeneous element $F$ of degree $\wss$, which is not a zero divisor, can be associated to an RRC system. In this case, it is either a canonical RC algebra, or it is a sub-RC algebra of a canonical RC algebra. We consider the following two cases.
\begin{description}
  \item[(I)] If $F$ satisfies the conditions $[F,M]_1\subseteq M\cdot F$ and $[F,F]_2\in M\cdot F^2$, then  for any $f\in M_k$, with $k\in \zn$, we define:
\[
  \ew4\:=\frac{[F,F]_2}{\wss^2(\wss+1)F^2}\in M_4\,, \ \ \ \rstdo f:=\frac{[F,f]_1}{\wss F}\in{M}_{k+2} \,, \ \ \ \qme:= \frac{\rcdo F}{\wss F},
\]
where $\rcdo$ is defined by \eqref{eq der zag}. With these definitions, we transform $M$ into a canonical RC algebra. In the corresponding RRC system we find:
\[
\prs_1(\vs_2,\vs_3,\ldots,\vs_\nvs)=\frac{[F,F]_2}{\wss^2(\wss+1)F^2}=\ew4, \ \ \prs_j(\vs_2,\vs_3,\ldots,\vs_\nvs)=\frac{[F,\vs_j]_1}{\wss F}, \ \ 2\leq j \leq \nvs,
\]
where $\vs_2,\vs_3,\ldots,\vs_\nvs$ are generators of $M$ and $t_1=\qme$.
  \item[(II)] If either of the  conditions in {\bf (I)} does not hold, then $M$ is a sub-RC algebra of the canonical RC algebra $\widehat{M}:=M\left[\vs_{\nvs+1}\right]$, where $\vs_{\nvs+1}:=\frac{1}{F}$. Since $[F,F]_1=0$, we get that $\rstdo F=0$, and thus $\rcdo \vs_{\nvs+1}=-\wss \vs_1 \vs_{\nvs+1}$.
\end{description}

These results hold thanks to Theorem \ref{thm main1}, as the degree of $\frac{1}{F}$ may be a negative integer.
For more details see \cite[Section 6]{zag94} and \cite[Remark 2.2]{bn23}.

The key takeaway is that, if we know $[F,\vs_j]_1$ for all $j=2,3,\ldots,\nvs$, and $[F,F]_2$, then we can determine all the RC brackets $[\cdot,\cdot]_\ast$ on $M$. This is because we can determine $\rcdo$ and we have $[\cdot,\cdot]_n=[\cdot,\cdot]_{\rcdo,n}$, for all non-negative integers $n$.

\bigskip

{\bf Acknowledgment.} The author completed the main results of this work during his one-year visit to the Max Planck Institute for Mathematics (MPIM) in Bonn in 2021. He would like to express his gratitude to MPIM. The author also thanks the ``Fundação Carlos Chagas Filho de Amparo à Pesquisa do Estado do Rio de Janeiro (FAPERJ)" for their support. While preparing the final version of this paper, he benefited from the ``Auxílio
Básico à Pesquisa (APQ1) - 2023"  process no. 210.604/2024 - SEI-260003/005867/2024.

\def\cprime{$'$} \def\cprime{$'$} \def\cprime{$'$}



\begin{thebibliography}{GMP95}

\bibitem[BN23]{bn23}
        G. Bogo,  Y. Nikdelan,
        \newblock Ramanujan systems of Rankin–Cohen type and hyperbolic triangles.
        \newblock {\em Forum Math}, 35(2023), no. 6, 1609–1629.

\bibitem[BGHZ08]{bhgz}
        J. H. Bruinier, G. van der Geer, G. Harder, and D. Zagier,
        \newblock  The 1-2-3 of Modular Forms,
        \newblock Universitext. Springer-Verlag, Berlin, 2008. Lectures from the Summer School on Modular Forms
         and their Applications held in Nordfjordeid, June 2004, Edited by Kristian Ranestad.

\bibitem[Coh75]{coh75}
        H. Cohen,
        \newblock Sums involving the values at negative integers of L--functions of quadratic characters.
        \newblock {\em Math. Ann.}, 217(1975), no. 3, 271--285.

\bibitem[CS17]{cs17}
        H. Cohen,  F. Str\"{o}mberg,
        \newblock Modular Forms: A Classical Approach.
        \newblock { American Mathematical Society, Providence, Rhode Island}, 2017.

\bibitem[KK06]{KK06}
        M. Kaneko,  M. Koike,
        \newblock On extremal quasimodular forms.
        \newblock {\em Kyushu J. Math.}, 60(2006), 457--470.

\bibitem[Kuz75]{kuz75}
        N. V. Kuznetsov,
        \newblock A new class of identities for the Fourier coefficients of modular forms.
        \newblock {\em Acta Arith.}, 27(1975), 505--519.


\bibitem[Mai11]{mai11}
	R. S. Maier,
	\newblock Nonlinear differential equations satisfied by certain classical modular forms.
	\newblock {\em Manuscripta Mathematica}, 134(2011), 1--42.

\bibitem[Mar96]{yvma}
	Y. Martin,
	\newblock Multiplicative $eta$-quotients.
	\newblock {\em Trans. Amer. Math. Soc.}, 348(1996), no. 12, 4825--4856.

\bibitem[MR09]{maro}
	F. Martin, E. Royer,
	\newblock Rankin-Cohen brackets on quasimodular forms.
	\newblock {\em J. Ramanujan Math. Soc.}, 24(2009), no. 3, 213--233.

\bibitem[Mov12b]{ho14}
    H. Movasati,
    \newblock Quasi modular forms attached to elliptic curves, {I}.
    \newblock {\em Annales Mathématique Blaise Pascal}, 19(2012), 307--377.

\bibitem[MN21]{movnik}
    H. Movasati, Y. Nikdelan,
    \newblock Gauss-{M}anin {C}onnection in {D}isguise: {D}work-{F}amily.
    \newblock {\em     J. Differential Geometry},  119(2021), no. 1, 73--98.

\bibitem[Nik20]{younes2}
        Y. Nikdelan,
        \newblock Modular vector fields attached to {D}work family: $\mathfrak{sl}_2(\C)$ Lie algebra.
        \newblock {\em Moscow Math. J.}, 20(2020), no. 1, 127--151.

\bibitem[Nik22]{younes5}
        Y. Nikdelan,
        \newblock Ramanujan-type systems of nonlinear ODEs for $\Gamma_0(2)$ and $\Gamma_0(3)$.
        \newblock {\em Expositiones Mathematicae}, 40(2022), no. 3, 409-431.

\bibitem[Nik24]{younes3}
        Y. Nikdelan,
        \newblock Rankin-Cohen brackets for {C}alabi-{Y}au modular forms.
        \newblock {\em     Commun. Number Theory Phys.}, 18(2024), no. 1, 1--48.

\bibitem[Ram16]{ra16}
        S. Ramanujan,
        \newblock On certain arithmetical functions.
        \newblock {\em Trans. Cambridge Philos. Soc.}, 22(1916), 159--184.

\bibitem[Ran56]{ran56}
        R.~A.~Rankin,
        \newblock The construction of automorphic forms from the derivatives of a given form.
        \newblock {\em J. Indian Math. Soc.}, 20(1956), 103--116.

\bibitem[WZ22]{wz22}
        W. Wang, H. Zhang,
        \newblock Meromorphic quasi-modular forms and their L-functions.
        \newblock {\em J. Number Theory}, 241(2022), 465--503.

\bibitem[Zag94]{zag94}
        D. Zagier,
        \newblock Modular forms and differential operators.
        \newblock {\em Proceedings Mathematical Sciences}, 104(1994), no. 1, 57--75.

\end{thebibliography}
\end{document}